\documentclass[12pt]{article}
\usepackage{amsmath,amssymb}
\usepackage{amsthm}
\usepackage{hyperref} 
\usepackage[pdftex]{graphicx}
\usepackage{pgf,tikz}
\usepackage{mathrsfs}
\usetikzlibrary{arrows}
\usepackage{multirow}        
\theoremstyle{definition}

\newtheorem{theorem}[equation]{Theorem}

\newtheorem{corollary}[equation]{Corollary}

\newtheorem{lemma}[equation]{Lemma}

\newtheorem{note}[equation]{Note}

\newcommand{\R}{\mathbb R}
\newcommand{\N}{\mathbb N}
\newcommand{\Z}{\mathbb Z}
\newcommand{\fix}{\text{Fix\,}}

\begin{document}

\title{Every continuous action of a compact group on a uniquely arcwise connected continuum has a fixed point}
\author{Benjamin Vejnar\footnote{The work was supported by the grant GA\v CR 17-04197Y.}
\\ Charles University
\\Faculty of Mathematics and Physics
}

\date{\today}\maketitle

\begin{abstract}
We are dealing with the question whether every group or semigroup action (with some additional property) on a continuum (with some additional property) has a fixed point.
One of such results was given in 2009 by Shi and Sun. They proved that every nilpotent group action on a uniquely arcwise connected continuum has a fixed point. We are seeking for this type of results with e.g. commutative, compact or torsion groups and semigroups acting on dendrites, dendroids, $\lambda$-dendroids and uniquely arcwise connected continua. We prove that every continuous action of a compact or torsion group on a uniquely arcwise connected continuum has a fixed point. We also prove that every continuous action of a compact and commutative semigroup on a uniquely arcwise connected continuum has a fixed point.
\end{abstract}



\section{Introduction}

The most classical case in the fixed point theory deals with a continuous selfmap of a topological space. It was already proved in 1909 by Brouwer that every continuous selfmap of a nonempty compact convex subset of a Euclidean space has a fixed point. Topological spaces whose every continuous selfmap has a fixed point are said to have the \emph{fixed point property}. It is well known that dendrites, dendroids or even $\lambda$-dendroids have the fixed point property \cite{manka}. Uniquely arcwise connected continua have the fixed point property for homeomorphisms, but they lack the fixed point property with respect to all continuous selfmaps \cite{mohler}. Moreover even weakly chainable uniquely arcwise connected continua do not have the fixed point property \cite{sobolewski}.
The last positive result mentioned here was strengthen in the sense that uniquely arcwise connected continua have the fixed point property for arc preserving mappings \cite{fugatemohler}. To the contrary, by a famous result of Bellamy, there is a continuous selfmap of a tree-like continuum without fixed points \cite{Bellamy}.

In this paper we are dealing with a family of selfmaps of a continuum and we are seeking for conditions under which there is a common fixed point for all these mappings. Of course there is no chance to get such results unless the mappings are related somehow one with each other. Commutativity is one of such relations of two mappings. However, commutativity is stil not enough to find a common fixed point since there are two commuting mappings of a closed interval without a common fixed point as was proved in \cite{boyce} and independently in \cite{huneke}. 

If we consider a family of mappings on a space with a common fixed point it is clear that the semigroup generated by them has the same common fixed point. This is the reason why we are interested only in groups or semigroups. We can deal with the problem in a seemingly more general setting by considering group or semigroup continuous actions. This general approach needs to be rearranged if we want to handle with open and monotone mappings. In that case, when considering an action of a group or a semigroup $G$ on a space $X$ we identify the elements $g\in G$ with the selfmaps $x\mapsto gx$ of $X$.

In the second part of this paper we prove that compact group actions or torsion group actions always have a fixed point when acting on a uniquely arcwise connected continuum. This should be compared with a result in \cite{FugateMcLean} that every compact group action on a tree-like continuum has a fixed point. Also a similar result for torsion groups was obtained already in \cite{Monod} but only for actions on dendrites. Moreover we prove that a continuous action of a compact and commutative semigroup on a uniquely arcwise connected continuum has a fixed point too. This should be compared with the main result of \cite{shi} by which every nilpotent group action on a uniquely arcwise connected continuum has a fixed point.

The existence of fixed points was also studied for semigroups acting on continua by monotone or open mappings. The negative result in \cite{MohlerOversteegen} that there exists a monotone mapping and an open mapping on a uniquely arcwise connected continuum without a fixed point corresponds with the positive result in \cite{gray} that every commutative semigroup acting on a $\lambda$-dendroid by monotone mappings has a fixed point. It is an open problem posed in \cite{graysmith} whether every commutative semigroup has always a fixed point when acting on a dendrite by open mappings.

The problem of existence of a common fixed point is related to the question whether several mappings have a common value. In spite of that it is well known (and easy to prove) that two commuting mappings of the closed interval have a common value, it is not known whether two commuting selfmaps of a simple triod share such a point \cite[Question 1]{mcdowell}.

\section{Main results}
Our main results are formulated in Corollary \ref{compact}, Corollary \ref{torsion} and Theorem \ref{ccsemigroup}.
All the undefined notions like dendrites, dendroids, uniquely arcwise connected continua, $\lambda$-dendroids, tree-like continua, open or monotone maps can be found in the classical book \cite{nadler} focused on continuum theory.
We just note that every dendrite is a dendroid, every dendroid is a $\lambda$-dendroid, every dendroid is tree-like as well as uniquely arcwise connected and there is no relation between $\lambda$-dendroids and uniquely arcwise connected continua and between uniquely arcwise connected continua and tree-like continua.

For better orientation in the amount of results we are adding Table 1 with references which answer the question:
``Does there exist a fixed point under every action of a group or a semigroup on a continuum from the given class?" So for example one can read the entry ``+ \cite{isbell}" as: ``Every commutative group acting on a dendrite has a fixed point." The minus sign means the negation and the question mark means that up to our knowledge the problem is open.

\begin{table}[h]
\centerline{
\begin{tabular}{|| r  p{4.2cm} | l | l | l | l |l ||}
\hline\hline
{\bf } &  &      \rotatebox{90}{\parbox[b]{2cm}{\bf dendrite}}    &     \rotatebox{90}{\parbox[b]{2cm}{\bf dendroid}}     &      \rotatebox{90}{\parbox[b]{2cm}{\bf 1-arc.con}}     &      \rotatebox{90}{\parbox[b]{2cm}{\bf $\lambda$-dendroid}}  &  \rotatebox{90}{\parbox[b]{2cm}{\bf tree-like}}
\\
\hline\hline
\multirow{8}{*}{ \rotatebox{90}{group}}  &   $\Z$ & + & + & + \cite{mohler} & + & $-$ \cite{FMnote} \\
\cline{3-7}
   &  commutative  & + \cite{isbell} & + & + & ? & $-$ \\
\cline{3-7}
  & compact  & + & +  & + C~\ref{compact} & + &+ \cite{FugateMcLean} \\
\cline{3-7}
  & amenable  &  $-$ N~\ref{amen} & $-$ &  $-$ &  $-$ & $-$ \\
\cline{3-7}
  & solvable  &  $-$ \cite{shi} &  $-$ &  $-$ &  $-$ & $-$ \\
\cline{3-7}
  & nilpotent  & + & + & + \cite{shi} & ? & $-$ \\
\cline{3-7}
  & torsion  & + \cite{Monod} & + & + C~\ref{torsion} & ? & ?  \\
\hline
\multirow{ 10}{*}{ \rotatebox{90}{semigroup}}  & $\N$ & +  & + \cite{borsuk} &  $-$ \cite{young}  & + \cite{manka} & $-$ \cite{Bellamy}\\
\cline{3-7}
  & compact  &  $-$ N~\ref{compactsemigroup} &  $-$ &  $-$ &  $-$ & $-$ \\
\cline{3-7}
  & commutative  &  $-$ \cite{boyce}\cite{huneke} &  $-$ &  $-$ &  $-$ & $-$ \\
\cline{3-7}
  & compact commutative  & + & + & + T~\ref{ccsemigroup} & + & + T~\ref{ccsemigrouptree} \\
\cline{3-7}
  & $\N$, monotone mappings & + & + &  $-$ \cite{MohlerOversteegen} & + &  $-$ \\
\cline{3-7}
  & commutative, monotone mappings & + & + N~\ref{zobecneni} &  $-$ & + \cite{gray} & $-$ \\
\cline{3-7}
  & $\N$, open mappings & + & + &  $-$ \cite{MohlerOversteegen} & + & $-$ \\
\cline{3-7}
  & commutative, open mappings  & ? \cite{graysmith}  N~\ref{opentrees}& ?  &   $-$  & ?  & $-$ \\
\hline\hline
\end{tabular}
}
\caption{The existence of fixed points}
\end{table}

The following lemma is just a suitable reformulation of the fact that the identity is the only increasing homeomorphism of $[0,1]$ which is contained in a compact group of homeomorphisms of $[0,1]$ (compare with \cite[Lemma~1]{mitchell}).

\begin{lemma}\label{monseq}
Let $X$ be a uniquely arcwise connected continuum, $g$ a homeomorphism of $X$ and $a\in X$. Suppose that there is an arc $x,y$ such that the sequence of points $a, g(a), g^2(a), \dots$ is strictly monotone sequence in the arc $xy$ with respect to a natural order of $xy$.
Then $g$ is not contained in a compact subgroup of the group of all homeomorphisms.
\end{lemma}

\begin{proof}
Suppose for simplicity that the sequence $a, g(a), g^2(a), \dots$ is increasing and denote by $s\in xy$ its supremum. 
Suppose for contradiction that $g$ is contained in a compact group. Then there is some cluster point $h$ of the sequence $g, g^2, \dots$. So there is an increasing sequence of positive integers $(n_i)$ such that $g^{n_i}$ converges uniformly to $h$. 
It follows that $h(a)=\lim g^{n_i}(a)=s$ and also that
\[h(s)=\lim_{i\to\infty} \left(g^{n_i}\left(\lim_{j\to\infty} g^j(a)\right)\right)=\lim_{i\to\infty}\lim_{j\to\infty} g^{n_i+j}(a)=s.\]

 Now $h(a)=s$ and $h(s)=s$ which is a contradiction since $a\neq s$ and $h$ is a bijection.
\end{proof}

\begin{lemma}\label{fixcon}
Let $X$ be a uniquely arcwise connected continuum and let $g$ be a homeomorphism of $X$ which is contained in a compact subgroup of the group of all homeomorphisms. Then $\fix g$ is arcwise connected.
\end{lemma}

\begin{proof}
Suppose not.
Since moreover $\fix g$ is closed, we can find two distinct points $x,y\in X$ such that the open arc between the points $x$ and $y$ does not contain any fixed points. Since $g(x)=x$ and $g(y)=y$ and $X$ is uniquely arcwise connected it follows that $g$ maps the arc $xy$ onto itself. Take any point $z$ in the the open arc $xy$. Clearly $g(z)\neq z$ so either $g(z)<z$ or $g(z)>z$ with respect to a natural linear order on $xy$. In any case the sequence $z, g(z), g^2(z), \dots$ contradicts Lemma~\ref{monseq}.
\end{proof}

The following lemma is an analogue to the result of Helly for convex subsets of $\R^m$ (see e.g. \cite{rabin}).

\begin{lemma}\label{Helly}
Let $X$ be a uniquely arcwise connected continuum and let $S_1,\dots, S_n$ be its arcwise connected subsets such that $S_i\cap S_j\neq\emptyset$. Then $\bigcap S_i\neq\emptyset$.
\end{lemma}

\begin{proof}
For every pair $i, j\leq n$ choose a point $s_{i,j}\in S_i\cap S_j$. Denote by $T$ the smallest tree which contains all the points $s_{i, j}$ for $i, j\leq n$. Since the sets $S_i\cap T$ are connected subsets of $T$ we can replace $X$ by $T$ and $S_i$ by $S_i\cap T$. The conclusion follows easily.
\end{proof}

\begin{theorem}\label{compactgroup}
Let $G$ be a group every element of which is contained in a compact subgroup. Then every continuous action of $G$ on a uniquely arcwise connected continuum has a fixed point. Moreover the set of fixed points of the action is an arcwise connected continuum.
\end{theorem}

\begin{proof}
We want to prove that $\bigcap \{\fix g\colon g\in G \}$ is nonempty.
It suffices to show that the system $\{\fix g\colon g\in G \}$ has the finite intersection property.
First let us take $g, h\in G$ and suppose first that $\fix g\cap \fix h$ is empty.

It follows by \cite{mohler} that $\fix g$ and $\fix h$ are nonempty. Since $X$ is arcwise connected and the sets $\fix g$ and $\fix h$ are disjoint and closed, there are points $x\in \fix g$ and $y\in \fix h$ such that the open arc $xy$ intersects neither $\fix g$ nor $\fix h$.

For any $z\in X$ denote by $t(z)$ the unique point in the closed arc $xy$ such that the closed arc $zt(z)$ intersects $xy$ in just one point $t(z)$.
Moreover denote $U=\{z\in X\colon t(z)=x\}$, $V=\{z\in X\colon t(z)=y\}$, $W=\{z\in X\colon t(z)\neq x, t(z)\neq y\}$. In order to remember this notation see Figure~\ref{obr}.

Let us consider several possible cases. If $g(y)\in V$ then the sequence $y, g^{-1}y, \dots$ contained in the arc $xg(y)$ gives a contradiction with Lemma~\ref{monseq}.

If $g(y)\in W$ consider the only point $w\in xy$ such that $wg(y)\cap xy=\{w\}$. Since the arc $xy$ is mapped by $g$ to $xg(y)$ it follows that $g(w)$ is contained in $xg(y)$. Moreover $w\notin\fix g$ so either $g(w)$ or $g^{-1}(w)$ is in the open arc $xw$. We get a contradiction with Lemma~\ref{monseq} using the sequence $g(w), g^2(w), \dots$ in the first case and using the sequence $g^{-1}(w), g^{-2}(w)$ in the second case.

So by symmetry we do not need to deal with the cases when $h(x)\in U$ and $h(x)\in W$.
So the remaining case satisfies $g(y)\in U$ and $h(x)\in V$. 
It follows that $g(V\cup W)\subseteq U$ and $h(U\cup W)\subseteq V$.
Denote the elements of the orbit of $y$ by $y_{n}=(hg)^n(y_n)$ for $n\in\Z$. We are going to show that every finite subset of the orbit of $y$ with respect to $hg$ is contained in an arc and moreover the order given by this arc corresponds to the order of these points given by their indeces. In order to simplify the proof let us write $\left<a, b, c\right>$ instead of $b\in ac$ which means that the points $a,b,c$ lie on some arc in this order. Since $g(y)\in U$ it follows that $hg(y)\in V$. Hence $\left<x, y_0, y_1\right>$. By applying the homeomorphism $g$ and noting that $g(x)=x$ we get $\left<x, g(y_0), g(y_1)\right>$ and since $g(y_0), g(y_1)\in U$ it follows that also $\left<y_0, g(y_0), g(y_1)\right>$. By applying $h$ we get $\left<y_0, y_1, y_2\right>$. Since $hg$ is a homeomorphism we get that $\left<y_n, y_{n+1}, y_{n+2}\right>$ for every $n\in\Z$.

Let us denote $R^+=\bigcup_{n=0}^\infty y_0y_n$, $R^-=\bigcup_{n=0}^{-\infty} y_0y_n$ and $R=R^+\cup R^-$ and consider two subcases. The first one is that either $R^+$ or $R^-$ is contained in an arc.
Then we obtain again a contradiction with Lemma~\ref{monseq}. 
In the second case  neither $R^+$ nor $R^-$ is contained in an arc. In this part of the proof we follow the ideas of \cite{mohler}. By the Krylov-Bogolioubov theorem \cite[p. 152]{Walters} we can consider a $hg$-invariant probability measure $\mu$ on $X$. 
We claim that for every $z\in X$ there is a unique point $r(z)\in R$ such that $zr(z)\cap R$ is just a one point set. Indeed, the arc $y_0z$ contains only finitely many points from $\{y_n\colon n\in\Z\}$ (otherwise $R^+$ or $R^-$ would be contained in an arc) and thus there is some $n\in\N$ such that $R\cap y_0z \subseteq y_{-n}y_n$. We set $r(z)$ to be the only point for which $zr(z)\cap y_{-n}y_n=\{r(z)\}$.

Let $A_n=\{z\in X\colon r(z)\in y_ny_{n+1}, r(z)\neq y_{n}\}$.
For every $n\in\Z$ the set $A_n$ is analytic by \cite{mohler} and hence universally measurable \cite[Theorem 21.10]{Kechris}. Moreover $hg(A_n)=A_{n+1}$ and thus $\mu(A_n)=\mu(A_{n+1})$. It follows that the space $X$ is disjointly decomposed into the sets $A_n, n\in\Z$, all of them having the same measure. Hence $\mu(X)$ is either zero if $\mu(A_0)=0$ of infinite if $\mu(A_0)>0$. This is impossible since $\mu$ is a probability measure.

In every case we achieved a contradiction so it follows that $\fix g\cap \fix h$ is nonempty.

Now, suppose that $g_1,\dots, g_n\in G$ and consider the corresponding sets of fixed points $\fix g_1, \dots, \fix g_n$. By Lemma~\ref{fixcon} the sets $\fix g_i$ are arcwise connected and hence we can use Lemma~\ref{Helly} to obtain that $\bigcap \fix g_i$ is nonempty. Hence by compactness
$\bigcap_{g\in G} \fix g$ is nonempty. Thus there is a fixed point of the action of $G$ on $X$.

Since all the sets $\fix g$ are arcwise connected and $X$ is uniquely arcwise connected it follows that the set $\bigcap_{g\in G} \fix g$ is arcwise connected as well.
\end{proof}

\begin{figure}[h]
\definecolor{uuuuuu}{rgb}{0.27,0.27,0.27}
\begin{tikzpicture}[line cap=round,line join=round,>=triangle 45,x=1.0cm,y=1.0cm]
\clip(-4.3,-2.1) rectangle (9.1,2.1);
\draw [line width=1.2pt] (0.,0.)-- (5.,0.);
\draw [line width=0.8pt,dash pattern=on 2pt off 2pt] (7.,0.) circle (2.cm);
\draw [line width=0.8pt,dash pattern=on 2pt off 2pt] (-2.,0.) circle (2.cm);
\draw [rotate around={0.:(2.5,0.)},line width=0.8pt,dash pattern=on 2pt off 2pt] (2.5,0.) ellipse (2.5cm and 2.cm);
\draw [line width=1.2pt] (5.,0.)-- (7.,1.);
\draw [line width=1.2pt] (5.,0.)-- (8.,0.);
\draw [line width=1.2pt] (5.,0.)-- (7.,-1.);
\draw [line width=1.2pt] (0.,0.)-- (-2.,1.);
\draw [line width=1.2pt] (0.,0.)-- (-3.,0.);
\draw [line width=1.2pt] (0.,0.)-- (-2.,-1.);
\draw [line width=1.2pt] (2.,0.)-- (2.,1.);
\draw [line width=1.2pt] (3.,0.)-- (3.,-1.);

\draw [fill=black] (0.,0.) circle (2.5pt);
\draw [fill=black] (5.,0.) circle (2.5pt);
\draw[color=black] (0.3,0.3) node {$x$};
\draw[color=black] (4.7,0.3) node {$y$};
\draw[color=black] (6.2,1.5) node {$V$};
\draw[color=black] (-2.8,1.5) node {$U$};
\draw[color=black] (1.5,1.5) node {$W$};
\end{tikzpicture}
\caption{Uniquely arcwise connected continuum $X$}\label{obr}
\end{figure}
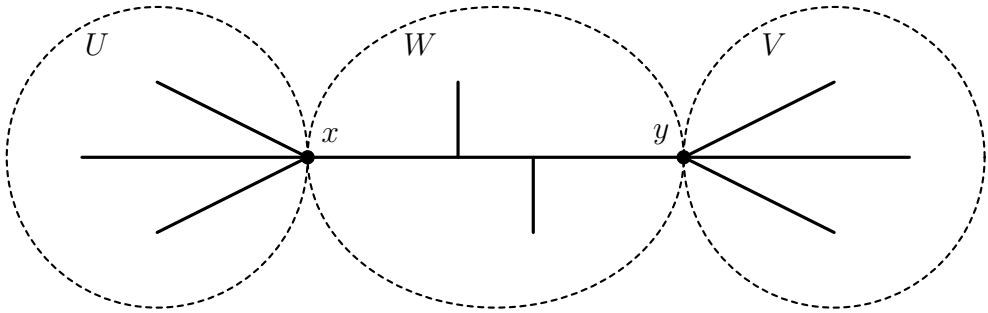

We can easily derive the subsequent two corollaries of Theorem~\ref{compactgroup}.

\begin{corollary}\label{compact}
Every compact group action on a uniquely arcwise connected continuum has a fixed point.
\end{corollary}

\begin{corollary}\label{torsion}
Every torsion group action on a uniquely arcwise connected continuum has a fixed point.
\end{corollary}

The following lemma was proved in \cite{mitchell}. We are using some more suitable notation for our purposes. It will allow us to pass from a semigroup action to a group action. Recall that a semigroup $G$ is said to be \emph{left reversible} if for every pair $a, b\in G$ there exists $c, d\in G$ such that $ac=bd$. Immediately one gets that commutative semigroups are left reversible.

\begin{lemma}\label{mitch}
Let $X$ be a compact Hausdorff space and let $G$ be a left reversible and equicontinuous semigroup of selfmaps of $X$. Then there is a retract $Y$ of $X$ and a compact subgroup $H$ of the closure of $G$ such that $H$ restricted to $Y$ is a group of homeomorphisms of $Y$ onto itself. Further the elements of $G$ have a common fixed point $p\in X$ if and only if $p\in Y$ and $p$ is a common fixed point of $H$.
\end{lemma}

\begin{theorem}\label{ccsemigroup}
Let $G$ be a compact commutative semigroup. Then every continuous action of $G$ on a uniquely arcwise connected continuum has a fixed point.
\end{theorem}

\begin{proof}
We can suppose that $G$ is a semigroup of selfmaps of a uniquely arcwise connected continuum $X$. Since $G$ is compact, it is equicontinuous and since $G$ is commutative, it is left reversible.
Hence we can apply Lemma~\ref{mitch} to get a retract $Y$ of $X$ and a subgroup $H$ of $G$ such that $H$ restricted to $Y$ is a group of homeomorphism of $Y$ and the action of $H$ on $Y$ has a fixed point if and only if the action of $G$ on $X$ has a fixed point.

Clearly $Y$ is a uniquely arcwise connected continuum since it is a retract of a uniquely arcwise connected continuum. Since $G$ is commutative the group $H$ is commutative and in particular nilpotent. Thus, by using \cite{shi} we conclude that the action of $H$ on $Y$ has a fixed point. Hence, the action of $G$ on $X$ has a fixed point.
\end{proof}

\begin{theorem}\label{ccsemigrouptree}
Let $G$ be a compact commutative semigroup. Then every continuous action of $G$ on a tree-like continuum has a fixed point.
\end{theorem}

\begin{proof}
We can follow similar concept as in the proof of Theorem~\ref{ccsemigroup}.
We apply Lemma~\ref{mitch} to get a retract $Y$ of $X$ and a compact subgroup $H$ of $G$ such that $H$ restricted to $Y$ is a group of homeomorphism of $Y$ and the action of $H$ on $Y$ has a fixed point if and only if the action of $G$ on $X$ has a fixed point.

Clearly $Y$ is a tree-like continuum since it is a retract of a tree-like continuum. Thus, by \cite{FugateMcLean} the action of $H$ on $Y$ has a fixed point. Hence, the action of $G$ on $X$ has a fixed point.
\end{proof}

We note that neither commutativity nor compactness can be omitted in Theorem~\ref{ccsemigroup} and Theorem~\ref{ccsemigrouptree}. We conclude this paper by several notes that are related to Table~1 and which contain some known generalizations or ideas for a possible proof.

\begin{note}\label{amen}
Since compact as well as commutative group actions on dendroids have fixed points, one could suspect that a similar result could hold for amenable groups. However this is not true even for actions on dendrites. Since solvable groups are amenable we have by a note in \cite{shi} that there is an amenable group action on a dendrite (even on an arc) without a fixed point. On the other hand, it was proved that in the case of countable amenable group actions on a dendrite \cite{shiamen}, or even on a uniquely arcwise connected continuum \cite{ShiYeAmen} there is always a fixed point or a 2-periodic point.
\end{note}

\begin{note}\label{compactsemigroup}
Two different constant mappings of an interval clearly form a compact semigroup with no common fixed point.
\end{note}

\begin{note}\label{zobecneni}
Every commutative semigroup acting by monotone mappings on a $\lambda$-dendroid has a fixed point \cite{gray}. This was
generalized in the case of dendroids in such a way that if we have a semigroup of monotone mappings and one more mapping that is commuting with all the others, then there is a common fixed point for all the mappings under discussion  \cite{graysmith}.
\end{note}

\begin{note}\label{opentrees}
In \cite{Bugati} it was proved that commutative semigroup of open mappings of a tree always has a fixed point. A possible way for the generalization to dendrites could be the use of a description of open mappings obtained in \cite{openmaps}.
\end{note}



\section{Acknowledgements}
I am grateful to J. Boro\'nski and R. Ma\'nka for their comments to the first version of this paper.

\bibliographystyle{siam}
\bibliography{citace}

\end{document}